\documentclass[reqno,centertags, 12pt]{amsart}
\usepackage{amsmath,amsthm,amscd,amssymb}
\usepackage{latexsym,verbatim}

\usepackage{graphicx,epsf}

-\marginparwidth \oddsidemargin -4mm \evensidemargin \oddsidemargin

\newcommand{\bbN}{{\mathbb{N}}}
\newcommand{\bbR}{{\mathbb{R}}}

\newcommand{\bbP}{{\mathbb{P}}}
\newcommand{\bbE}{{\mathbb{E}}}
\newcommand{\bbZ}{{\mathbb{Z}}}

\newcommand{\bbT}{{\mathbb{T}}}

\newcommand{\lb}{\label}

\newcommand{\supp}{\text{\rm{supp}}}

\newcommand{\beq}{\begin{equation}}
\newcommand{\eeq}{\end{equation}}
\newcommand{\ba}{\begin{align}}
\newcommand{\ea}{\end{align}}
\newcommand{\eps}{\varepsilon}

\newcommand{\tht}{\theta}

\newcommand{\til}{\tilde}

\newcounter{smalllist}
\newenvironment{SL}{\begin{list}{{\rm\roman{smalllist})}}{%
\setlength{\topsep}{0mm}\setlength{\parsep}{0mm}\setlength{\itemsep}{0mm}%
\setlength{\labelwidth}{2em}\setlength{\leftmargin}{2em}\usecounter{smalllist}%
}}{\end{list}}

 \allowdisplaybreaks
\numberwithin{equation}{section}

\newtheorem{theorem}{Theorem}[section]

\newtheorem{lemma}[theorem]{Lemma}
\newtheorem{corollary}[theorem]{Corollary}
\theoremstyle{definition}

\newtheorem{example}[theorem]{Example}

\theoremstyle{remark}

\begin{document}
\title[Front Speed-up and Quenching]
{Pulsating Front Speed-up and Quenching of Reaction \\ by Fast
Advection}

\author{Andrej Zlato\v s}

\address{\noindent Department of Mathematics \\ University of
Chicago \\ Chicago, IL 60637, USA \newline Email: \tt
zlatos@math.uchicago.edu}


\begin{abstract}
We consider reaction-diffusion equations with combustion-type non-linearities in two dimensions and study speed-up of their pulsating fronts by general periodic incompressible flows with a cellular structure. We show that the occurence of front speed-up in the sense $\lim_{A\to\infty} c_*(A)=\infty$, with $A$ the amplitude of the flow and $c_*(A)$ the (minimal) front speed, only depends on the geometry of the flow and not on the reaction function. In particular, front speed-up occurs for KPP reactions if and only if it does for ignition reactions. We provide a sharp characterization of the periodic symmetric flows which achieve this speed-up and also show that these are precisely those which, when scaled properly, are able to quench any ignition reaction.
\end{abstract}

\maketitle

\section{Introduction and Examples} \lb{S1}

In this paper we study the effects of strong incompressible advection
on combustion. We consider the reaction-advection-diffusion equation
\begin{equation} \lb{1.1}
T_t + Au(x)\cdot\nabla T = \Delta T + f(T), \qquad
T(0,x)=T_0(x)\in[0,1]
\end{equation}
on $D\equiv\bbR\times\bbT^{d-1}$, with $u$ a prescribed flow profile
and $A\gg 1$ its amplitude. Here $T(t,x)\in[0,1]$ is the normalized
temperature of a premixed combustible gas and $f$ is the burning
rate.

We assume that $u\in C^{1,\eps}(D)$ is a periodic incompressible
(i.e., $\nabla \cdot u\equiv 0$) vector field which is symmetric
across the hyperplane $x_1=0$. That is, $u(Rx)=Ru(x)$ where
$R(x_1,\dots,x_d)=(-x_1,x_2,x_3,\dots,x_d)$ is the reflection across
$x_1=0$. If the period of $u$ in $x_1$ is $p$, then this implies
that $u$ is symmetric across each hyperplane $x_1=k p$, $k\in\bbZ$.
Hence $u$ is a periodic symmetric flow of {\it cellular type} (since
$u_1(x)=0$ when $x_1 \in p\bbZ$) with $[0,p]\times\bbT^{d-1}$ a cell
of periodicity.

The reaction function $f\in C^{1,\eps}([0,1])$ is of {\it combustion
type}. That is, there is $\tht_0\in[0,1)$ such that $f(s)=0$ for
$s\in[0,\tht_0]\cup\{1\}$ and $f(s)>0$ for $s\in(\tht_0,1)$, and $f$
is non-increasing on $(1-\eps,1)$ for some $\eps>0$. This includes
the {\it ignition reaction} term with $\tht_0>0$ and {\it positive
reaction} term with $\tht_0=0$. In the latter case we single out the
{\it Kolmogorov-Petrovskii-Piskunov (KPP) reaction} \cite{KPP} with
$0<f(s)\le sf'(0)$ for all $s\in(0,1)$.

We will be interested in two effects of the strong flow $Au$ on
combustion: pulsating front speed enhancement and quenching of
reaction. This problem has recently seen a flurry of activity ---
see
\cite{ABP,B,BHN-2,CKOR,CKR,CKRZ,FKR,Heinze,KR,KZ,NR,RZ,ZlaArrh,ZlaMix,ZlaPercol}.
A {\it pulsating front} is a solution of \eqref{1.1} of the form
$T(t,x)=U(x_1-ct,x)$, with $c$ the front speed and $U(s,x)$
periodic in $x_1$ (with period $p$) such that
\[
\lim_{s\to-\infty}U(s,x)=1 \qquad\text{and}\qquad
\lim_{s\to+\infty}U(s,x)=0,
\]
uniformly in $x$. It is well known \cite{BHN-1} that in the case of
positive reaction there is $c_*(A)$, called the {\it minimal
pulsating front speed}, such that pulsating fronts exist precisely
for speeds $c\ge c_*(A)$. In the ignition reaction case the front
speed is unique and we again denote it $c_*(A)$. In the present
paper we will be interested in the enhancement of this (minimal)
front speed by strong flows.

We say that the flow $Au$ {\it quenches} (extinguishes) the initial
``flame'' $T_0$ if the solution of \eqref{1.1} satisfies
$\|T(t,\cdot)\|_\infty \to 0$ as $t\to\infty$. Here one usually
considers compactly supported initial data. The flow profile $u$ is
said to be {\it quenching} for the reaction $f$ if for any compactly
supported initial datum $T_0$ there is an amplitude $A_0$ such that
$T_0$ is quenched by the flow $Au$ whenever $A\ge A_0$. We note that
quenching never happens for KPP reactions --- the solutions of
\eqref{1.1} for compactly supported non-zero $T_0$ always propagate
and the speed of their spreading equals $c_*(A)$ \cite{BHN-1,
Weinberger}.


In this paper we characterize those periodic symmetric
incompressible flows in two dimensions which achieve speed-up of
fronts and, if scaled properly, quenching of any ignition reaction.
For $l>0$ we denote by $l\bbT$ the interval $[0,l]$ with its ends
identified, and we let $u^{(l)}(x)\equiv u(x/l)$ be the scaled flow
on $\bbR\times l\bbT$ (with cells of size $lp\times l$).

\begin{theorem}\lb{T.1.1}
Let $u$ be a $C^{1,\eps}$ incompressible $p$-periodic flow on
$D=\bbR\times\bbT$ which is symmetric across $x_1=0$, and let $f$ be
any combustion-type reaction.
\begin{SL}
\item[{\rm{(i)}}] If the equation
\begin{equation}\label{1.2}
u\cdot\nabla\psi=u_1
\end{equation}
on $p\bbT\times\bbT$ has a solution $\psi\in H^1(p\bbT\times\bbT)$,
then
\begin{equation}\label{1.3}
\limsup_{A\to\infty} c_*(A)<\infty
\end{equation}
and no $u^{(l)}$ is quenching for $f$.
\item[{\rm{(ii)}}] If \eqref{1.2} has no $H^1(p\bbT\times\bbT)$-solutions, then
\begin{equation}\label{1.4}
\lim_{A\to\infty} c_*(A)=\infty
\end{equation}
and if $f$ is of ignition type, then there is $l_0\in(0,\infty)$
such that the flow $u^{(l)}$ on $\bbR\times l\bbT$ is quenching for
$f$ when $l<l_0$ and not quenching when $l>l_0$.
\end{SL}
\end{theorem}

\noindent
{\it Remarks.} 1. The proof shows that in (ii), $l_0\ge c\|f(s)/s\|_\infty^{-1/2}$ for some $u$-independent $c>0$.
It can also be showed that the claim $l_0>0$ in (ii) extends to some positive reactions
that are weak at low temperatures (more precisely, $f(s)\le \alpha
s^{\beta}$ for some $\alpha>0$ and $\beta>3$ --- see Corollary~\ref{C.5.2}),
in particular, the {\it Arrhenius} reaction
$f(s)=e^{-C/s}(1-s)$, $C>0$. On the other hand, if $f(s)\ge \alpha
s^{\beta}$ for some $\alpha>0$, $\beta<3$, and all small $s$, then $l_0=0$ for any $u$ \cite{ZlaArrh}.
\smallskip

2. We note that $l_0=\infty$ is impossible for cellular flows in two
dimensions --- see \cite{ZlaMix} which studies {\it strongly
quenching flows} $u$, that is, quenching for any
ignition reaction and any $l$.
\smallskip

3. Although we only consider periodic boundary conditions here, it
is easy to see that Theorem \ref{T.1.1} remains valid for
\eqref{1.1} on $\bbR\times[0,1]$ with Neumann boundary conditions,
provided $u_2(x)=0$ when $x_2\in\{0,1\}$.
\smallskip

4. Although a part of our analysis --- Sections \ref{S2} and \ref{S3} --- is
valid in any dimension, it remains an open quenstion whether Theorem \ref{T.1.1}
also extends beyond two dimensions.
\smallskip

Theorem \ref{T.1.1} has the following corollary:

\begin{corollary} \lb{C.1.2}
Let $u$ be a $C^{1,\eps}$ incompressible $p$-periodic flow on
$D=\bbR\times\bbT$ which is symmetric across $x_1=0$. Then speed-up of pulsating
fronts by $u$ in the sense of \eqref{1.4} occurs for ignition reactions if
and only if it occurs for KPP reactions.
\end{corollary}

{\it Remark.} Although speed-up of KPP fronts has been studied extensively (see, e.g., \cite{B,BHN-2,CKOR,Heinze,KR,NR,RZ,ZlaPercol}), rigorous results on ignition front speed-up have so far been established only in two dimensions for percolating flows and special cellular flows \cite{KR} (see below).
\smallskip

It is not surprising that the flows which achieve speed-up of fronts
are precisely those which quench large initial data. Fast fronts are
long, the latter being due to \hbox{short time--long distance}
mixing by the underlying flow. Such mixing yields quenching,
although possibly only away from regions where the flow is
relatively still (e.g., the centers of the cells in
Figure~\ref{fig-cell} below). If these regions are sufficiently
small, for instance when the flow is scaled, then reaction cannot
survive inside them and global quenching follows. This relation of
front speed to flow mixing properties also illuminates Corollary
\ref{C.1.2}.

Note that the above assumptions on $u$ exclude the class of {\it
percolating flows} (in particular, {\it shear flows}
$u(x)=(\alpha(x_2,\dots,x_d),0,\dots,0)$) which possess streamlines
connecting $x_1=-\infty$ and $x_1=+\infty$. In two dimensions, the
conclusions of Theorem \ref{T.1.1}(ii) for these flows have been
established in \cite{CKOR,CKR,KR,KZ,RZ}. Moreover, results from
\cite{BHN-2,ZlaPercol} can be used to prove linear pulsating front
speed-up (namely, $\lim_{A\to\infty} c_*(A)/A>0$) by percolating
flows in the presence of KPP reactions in any dimension.

As for cellular flows in two dimensions (the kind we consider here),
the claims about the front speed $c_*(A)$ in Theorem \ref{T.1.1}
have been proved for KPP reactions in \cite{RZ}. The special case of
the flow $u(x)=\nabla^\perp H (x) \equiv (-H_{x_2},H_{x_1})$ with
the {\it stream function} $H(x_1,x_2)=\sin 2\pi x_1 \sin 2\pi x_2$
has been addressed in \cite{FKR,KR,NR}, which proved \eqref{1.4} for
any reaction and quenching by $u^{(l)}$ for small enough $l$ and
ignition reactions. The streamlines of this flow are depicted in
Figure~\ref{fig-cell}.

\begin{figure}[ht!]
 \centerline{\epsfxsize=0.36\hsize \epsfbox{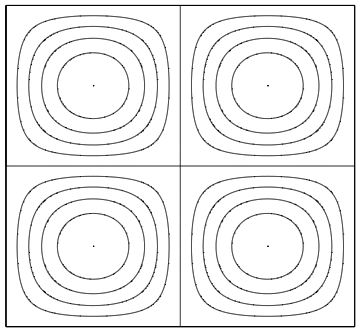}}
 \caption{A cellular flow.}
 \label{fig-cell}
\end{figure}

We note that it is easy to show that \eqref{1.2} has no
$H^1(\bbT^2)$-solutions in this case \cite{RZ}, and so one can
recover these results from Theorem \ref{T.1.1}(ii). Our general
method does not yield the more precise asymptotics $c_*(A)\sim
A^{1/4}$ in the KPP case \cite{NR} and $A^{1/5}\lesssim c_*(A)
\lesssim A^{1/4}$ in the ignition case \cite{KR} for this particular
flow.


We conclude this introduction with two more examples of types of
flows to which Theorem~\ref{T.1.1} applies.

\begin{example} \label{E.1.3} {\bf Checkerboard flows.}
Consider the cellular flow above vanishing in every other cell as
depicted in Figure~\ref{fig-checker}, thus forming a
checkerboard-like pattern. This flow is both periodic (with period
2) and symmetric but it is not $C^{1,\eps}$. Let us remedy this
problem by letting the stream function be $H(x_1,x_2)=(\sin 2\pi x_1
\sin 2\pi x_2)^\alpha$ with $\alpha>2$ in the cells where $u$ does
not vanish. Again, \eqref{1.2} has no
$H^1(2\bbT\times\bbT)$-solutions \cite{RZ}, and so Theorem
\ref{T.1.1}(ii) --- speed-up of fronts and quenching by $u^{(l)}$
--- holds. Moreover, the same conclusion is valid for other flows
with this type of structure, even if the angle of contact of the
``active'' cells is $\pi$.
\end{example}

\begin{figure}[ht!]
 \centerline{\epsfxsize=0.36\hsize \epsfbox{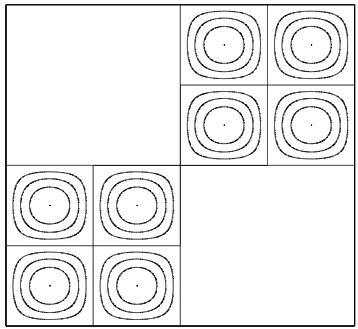}} 
 \caption{A checkerboard cellular flow.}
 \label{fig-checker}
\end{figure}

\begin{example} \label{E.1.4} {\bf Flows with gaps.}
Consider again the cellular flow above but with a vertical ``gap''
of width $\delta>0$, in which the flow vanishes, inserted in place
of each vertical segment $\{k\}\times\bbT$, $k\in\bbZ$, such as
shown in Figure~\ref{fig-gaps}. We again need to alter the stream
function as we did in the previous example in order to make the flow
$C^{1,\eps}$. This time it is easy to see that \eqref{1.2} has
$H^1((1+\delta)\bbT\times\bbT)$-solutions \cite{RZ}, and so Theorem
\ref{T.1.1}(i) --- no speed-up of fronts and no quenching by
$u^{(l)}$
--- holds in this case. The same conclusion is valid for other flows with similar
structures of streamlines, even when the gaps are replaced
by channels in which the flow moves ``along'' the channel only (see
\cite{RZ} for more details).

We also note that Sections \ref{S2} and \ref{S3} below yield the conclusions of Theorem \ref{T.1.1}(i) for cellular flows with gaps in any dimension (using that gaps force Lemma \ref{L.2.2}(ii) to hold).
\end{example}

\begin{figure}[ht!]
 \centerline{\epsfxsize=0.36\hsize \epsfbox{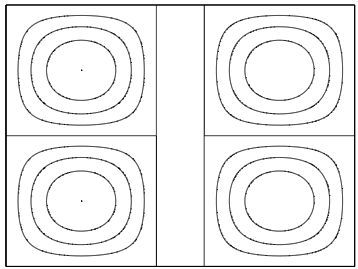}}
 \caption{A cellular flow with gaps.}
 \label{fig-gaps}
\end{figure}

The rest of the paper consists of Section~\ref{S2} where we prove a
few preliminary lemmas, and Sections \ref{S3} and \ref{S4} which
contain the proof of Theorem \ref{T.1.1}.

The author would like to thank Sasha Kiselev, Tom Kurtz, and Greg Lawler for useful
discussions. Partial support by the NSF through the grant
DMS-0632442 is also acknowledged.

\section{Some Preliminaries} \lb{S2}

In this and the next two sections we will assume the hypotheses of
Theorem \ref{T.1.1} with the period $p=1$ --- the general case is
handled identically. This implies that $u$ is symmetric across each
hyperplane $x_1=k$, $k\in\bbZ$. The analysis in this section and the
next applies to \eqref{1.1} on $D=\bbR\times\bbT^{d-1}$ for any
$d\in\bbN$.

Let us consider the stochastic process $X_t^{A,x}$ starting at $x\in
D$ and satisfying the stochastic differential equation
\begin{equation} \lb{2.1}
dX^{A,x}_t=\sqrt{2}\,dB_t - Au(X^{A,x}_t)dt, \qquad X^{A,x}_0=x,
\end{equation}
where $B_t$ is a normalized Brownian motion on $D$. We note that by
Lemma 7.8 in \cite{Oks}, we have that if
\begin{equation} \lb{2.2}
\phi_t + Au(x)\cdot\nabla \phi = \Delta \phi, \qquad
\phi(0,x)=\phi_0(x),
\end{equation}
then
\begin{equation} \lb{2.3}
\phi(t,x) = \bbE \big( \phi_0(X^{A,x}_t) \big).
\end{equation}
In particular, $\phi_0(x)=\chi_{[-L,L]}(x)$ gives
\begin{equation} \lb{2.4}
\phi(t,x) = \bbP \big( |X^{A,x}_t|\le L \big),
\end{equation}
where we define $|x|\equiv |x_1|$ for $x\in D$. Also notice that if
$\phi_0=T_0\in[0,1]$, then by comparison theorems \cite{Sm} for any
$t,x$,
\begin{equation} \lb{2.4a}
0\le T_0(t,x)\le e^{t\|f(s)/s\|_\infty} \phi(t,x) \le
e^{t\|f'\|_\infty} \phi(t,x).
\end{equation}

\begin{lemma} \lb{L.2.1}
\begin{SL}
\item[{\rm{(i)}}] If $k\in\bbZ$ and $y_1=k$ then the distribution of
$X^{A,y}_t$ is symmetric across the hyperplane $x_1=k$, that is,
\[
\bbP(X^{A,y}_t\in V)=\bbP \big(X^{A,y}_t\in R(V-(k,0))+(k,0) \big)
\]
for each $V\subseteq D$.
\item[{\rm{(ii)}}] If $k\in\bbZ$ and $y_1\ge k$, then
for any $I\subseteq\bbR^+$,
\begin{equation} \lb{2.5}
\bbP \big( (X^{A,y}_t)_1 \in k+I \big) \ge \bbP \big( (X^{A,y}_t)_1
\in k-I \big).
\end{equation}
When $y_1\le k$, the inequality in \eqref{2.5} is reversed.
\item[{\rm{(iii)}}] If $L\in\bbN$, then
\begin{equation} \lb{2.6}
\bbP \big( |X^{A,y}_t|\le L \big) \le \left\lceil \frac{|y_1|}L
\right\rceil^{-1}.
\end{equation}
\end{SL}
\end{lemma}

\begin{proof}
(i) and (ii) are obvious from the symmetry of $u$ across $x_1=k$ and
from almost sure continuity of $X_t^{A,y}$ in $t$. To show (iii), it
is sufficient to consider $y_1>L$. Applying (ii) with $k=jL$ for
$j=1,\dots,\lceil y_1/L\rceil-1$, we see that
\[
\bbP \big( (X^{A,y}_t)_1\in [-L,L] \big) \le \bbP \big(
(X^{A,y}_t)_1\in [(2j-1)L,(2j+1)L] \big).
\]
The claim follows.
\end{proof}

Next we prove the following key dichotomy.


\begin{lemma} \lb{L.2.2}
For any sequence $\{A_n\}_{n=1}^\infty$ one of the following holds.
\begin{SL}
\item[{\rm{(i)}}] For any $t,\eps>0$ and $L<\infty$ there are
$x,n$ such that
\begin{equation} \lb{2.7}
\bbP \big(|X^{A_n,x}_t-x|\le L \big) < \eps.
\end{equation}
\item[{\rm{(ii)}}] For any $t,\eps>0$ there is $L<\infty$ such that
for any $x,n$,
\begin{equation} \lb{2.8}
\bbP \big(|X^{A_n,x}_t-x|\le L \big) > 1- \eps.
\end{equation}
\end{SL}
\end{lemma}

\begin{proof}
Let us first assume that there is $t'>0$ such that for any $\eps'>0$
and $L'<\infty$ there are $x$, $n$ such that
\begin{equation} \lb{2.9}
\bbP \big(|X^{A_n,x}_{t'}-x|\le L' \big) < \eps'.
\end{equation}
Given any $\eps>0$, $L\in\bbN$, let $m>2/\eps$ be an integer and let
 $x,n$ be as in \eqref{2.9} with $\eps'=1/m$, $L'=(2m+1)L$. Notice that by
periodicity of $u$ we can assume $|x_1|\le 1$. For any $t\ge t'$ we
have
\[
\bbP \big( |X^{A_n,x}_t-x|\le L \big)  \le \bbP \big(
|X^{A_n,x}_t|\le 2L \big) \le \bbP \big(|X^{A_n,x}_{t'}|\le 2mL
\big) + \sup_{|y|\ge 2mL}\bbP \big(|X^{A_n,y}_{t-t'}|\le 2L\big).
\]
The first term is smaller than $\eps'<\eps/2$ by \eqref{2.9} and the
second is at most $1/m<\eps/2$ by \eqref{2.6}. This yields (i) for
$t\ge t'$. On the other hand, if (i) does not hold for some
$t\in(0,t')$, then there are $\eps,L$ such that for all $x,n$,
\[
\bbP \big(|X^{A_n,x}_{t}-x|\le L \big) \ge \eps.
\]
Choose $m\in\bbN$ so that $mt\ge t'$. It follows that
\[
\bbP \big(|X^{A_n,x}_{mt}-x|\le mL \big) \ge \eps^m
\]
for all $x,n$. But this contradicts (i) for $mt$, which has just
been proven. Therefore (i) holds for all $t>0$ under the hypothesis
above.

Now assume the opposite case to the one above. Namely, that for each
$t'>0$ there are $\eps'>0$ and $L'<\infty$ such that for all $x$,
$n$,
\begin{equation} \lb{2.10}
\bbP \big(|X^{A_n,x}_{t'}-x|\le L' \big) \ge \eps'.
\end{equation}
We will show that then (ii) holds, thus finishing the proof.

For each $t>0$ let
\[
\eps_0(t)  \equiv \sup_{L} \inf_{x,n} \bbP \big(|X^{A_n,x}_{t}-x|\le
L \big) >0
\]
Periodicity of $u$ guaranties that
\[
\eps_0(t) =\sup_{L\in\bbN} \inf_{|x|\le 1,n} \bbP
\big(|X^{A_n,x}_{t}|\le L \big) \,\, \big(\equiv \sup_{L\in\bbN}
\eps_1(t,L) \big).
\]
Notice that $\eps_0(t)$ is non-increasing. Indeed,
for $L,m\in\bbN$ and  $t\ge t'$,
\begin{align} \lb{2.11}
\eps_1(t,L) & \le \eps_1(t',mL) + \frac 1m
\end{align}
by \eqref{2.6}, and so $\eps_0(t)\le \eps_0(t') + 1/m$ for any $m$.

We will now show that $\eps_0(t)=1$ for all $t$. To this end assume
$\eps_0(t)<1$ for some $t$. Let $m$ be large (to be chosen later),
and let $L$ be such that
\begin{align} \lb{2.12}
\eps_1(t,L)>\eps_0(t) - \frac 1m
\end{align}
Consider any $|x|\le 1$, $n$ such that
\begin{align} \lb{2.13}
\bbP \big(|X^{A_n,x}_{t}|\le (2m+1)L \big) \le \eps_0(t)+\frac 1m.
\end{align}
Such $x,n$ do exists because of $\eps_0(t)\ge\eps_1(t,(2m+1)L)$.
Then the set of Brownian paths for which there is $t'\in[0,t]$ such
that $|X^{A_n,x}_{t-t'}|=(m+1)L$ has measure at least
$1-\eps_0(t)-1/m$. Since
\begin{align*}
\bbP \big(|X^{A_n,x}_t| \in [L,(2m+1)L] \,\big|\, &
|X^{A_n,x}_{t-t'}| =(m+1)L \text{ for some $t'\in[0,t]$} \big)
\\ & \ge \inf_{t'\in[0,t]} \eps_1(t',mL)
> \eps_0(t) - \frac 2m
\end{align*}
by \eqref{2.11} and \eqref{2.12}, this means
\begin{align*}
\bbP \big(|X^{A_n,x}_t|\le (2m+1)L \big) & = \bbP
\big(|X^{A_n,x}_t|\le L \big) + \bbP \big(|X^{A_n,x}_t|\in
[L,(2m+1)L] \big)
\\ &\ge \eps_1(t,L) +
\bigg(1-\eps_0(t)-\frac 1m \bigg) \bigg(\eps_0(t) - \frac 2m \bigg)
\\ & \ge \bigg(2-\eps_0(t)-\frac 1m \bigg) \bigg(\eps_0(t) - \frac 2m \bigg).
\end{align*}
Since $\eps_0(t)<1$, this is larger than $\eps_0(t)+1/m$ when $m$ is
large enough. This, however, contradicts \eqref{2.13}. Therefore we
must have $\eps_0(t)=1$ for all $t$, which is (ii).
\end{proof}

We will also need the following result which is essentially from
\cite{CKRZ}.

\begin{lemma} \lb{L.2.3}
For any $d\in\bbN$, there is $c>0$ such that for any Lipschitz
incompressible flow $u$, any $A$, and any $t\ge 0$, the solution
$\phi$ of \eqref{2.2} on $\Omega\equiv [0,1]\times\bbT^{d-1}$ with
Dirichlet boundary conditions on $\partial\Omega$ satisfies
\begin{equation} \lb{2.14}
\|\phi(t,\cdot)\|_\infty \le 2e^{-ct} \|\phi_0\|_\infty.
\end{equation}
\end{lemma}

\begin{proof}
The maximum principle implies that it is sufficient to show that
there is $\tau>0$ such that
\[
\|\phi(\tau,\cdot)\|_\infty \le \frac 12 \|\phi_0\|_\infty.
\]
uniformly in $u$ and $A$. For incompressible flows on $\bbT^d$ and
mean-zero $\phi_0$ this follows from Lemma 5.6 in \cite{CKRZ}. The
proof extends without change to our case, the Dirichlet boundary
condition replacing the mean-zero assumption when the Poincar\'e
inequality is used.
\end{proof}

\section{Proof of Theorem \ref{T.1.1}: Part I} \lb{S3}

Let us now assume that $u$ and $f$ are as in Theorem \ref{T.1.1} and
$A_n\to\infty$ is such that Lemma \ref{L.2.2}(ii) holds. We will
then show that the minimal front speeds $c_*(A_n)$ are uniformly
bounded and the flows $A_nu$ do not quench large enough compactly
supported initial data $T_0$ for \eqref{1.1}. The analysis in this
section applies to $D=\bbR\times\bbT^{d-1}$ for any $d\in\bbN$.

\begin{lemma} \lb{L.3.1}
Consider the setting of Theorem \ref{T.1.1} with
$D=\bbR\times\bbT^{d-1}$, and let $A_n\to\infty$ be such that Lemma
\ref{L.2.2}(ii) holds. Then $c_*(A_n)$ are uniformly bounded above.
\end{lemma}

\begin{proof}
Choose $L\in\bbN$ that satisfies Lemma \ref{L.2.2}(ii) for $t=1$ and
$\eps=\tfrac 14$. Let $x$ be such that $x_1\in \bbZ$ and consider
$X_t^{A_n,x}$ from \eqref{2.1}. Take $\tau_0=0$ and let $\tau_j$ be
the first time such that
$|X_{\tau_j}^{A_n,x}-X_{\tau_{j-1}}^{A_n,x}|=3L$ (recall that
$|x|=|x_1|$). We then have from \eqref{2.8} and \eqref{2.6},
\[
\bbP(\tau_j-\tau_{j-1}\le 1)\le\frac 12
\]
because $\tfrac 13 p + (1-p)\ge \tfrac 34$ implies $p\le \tfrac 12$.
This means that for any large enough $C,t\in\bbN$,
\begin{align*}
\bbP(|X_t^{A_n,x}-x|\ge 3LCt)\le \bbP(\tau_{Ct}\le t) & \le
\sum_{j=0}^{t-1} {Ct \choose j} \bigg(\frac{1}{2}\bigg)^{Ct-j} \le
{Ct \choose t} \frac {t}{2^{(C-1)t}} \\ & \le \bigg(
\frac{(5/4)^{C-1}C^C}{2^{C-1}(C-1)^{C-1}} \bigg)^t \le \kappa(C)^t
\end{align*}
with $\kappa(C)\equiv 2Ce(2/3)^C\to 0$ as $C\to\infty$. We used here
the fact that fewer than $t$ of the differences $\tau_j-\tau_{j-1}$
can exceed 1 in the second inequality, and Stirling's formula in the
fourth.

Let now $T$ be the solution of \eqref{1.1} with $A=A_n$ and
$T_0\equiv \chi_{\bbR^-\times\bbT^{d-1}}$. If $\phi$ solves
\eqref{2.2} with $A=A_n$ and $\phi_0\equiv T_0$, then we have by \eqref{2.4a} for
$x(s)\equiv (s,0,\dots,0)$,
\[
T(t,x(3LCt))\le e^{t\|f'\|_\infty}\phi(t,x(3LCt)) \le
e^{t\|f'\|_\infty} \bbP(|X_t^{A_n,x(3LCt)}-x(3LCt)|\ge 3LCt) \to 0
\]
as $t\to\infty$, provided $C$ is large enough. On the other hand, it
is well known that $T(t,x(ct))\to 1$ as $t\to\infty$ when
$c<c_*(A_n)$ \cite{BHN-1,Weinberger,Xin}. This means $c_*(A_n)\le
3LC$ and we are done.
\end{proof}

\begin{lemma} \lb{L.3.2}
Consider the setting of Theorem \ref{T.1.1} with
$D=\bbR\times\bbT^{d-1}$, and let $A_n\to\infty$ be such that Lemma
\ref{L.2.2}(ii) holds. Then there is compactly supported
$T_0(x)\in[0,1]$ such that the solution $T$ of \eqref{1.1} with
$A=A_n$ does not quench for any $n$.
\end{lemma}

\begin{proof}
By comparison theorems, we only need to consider $f$ of ignition
type --- with $\tht_0>0$. We again choose $L\in\bbN$ that satisfies
Lemma \ref{L.2.2}(ii) for $t=1$ and $\eps=\tfrac 12$. We next note
that there is $\delta>0$ such that
\begin{equation} \lb{3.1}
\bbP(|X_t^{A_n,x}-x|\ge t^{8/15}) \le e^{-t^{\delta}}
\end{equation}
for all large enough $t$ and all $x\in D$ and $n$. Indeed, assume
$x_1\in \bbZ$ and $t\in\bbZ$ (the general case follows immediately
from this), and let $j(t)=\inf\{j\,|\,\tau_j> t\}$, with $\tau_j$
from the proof of Lemma \ref{L.3.1}. Then that proof shows that for
$C\in\bbZ$ we have
\begin{equation} \lb{3.2}
\bbP(j(t)> Ct) = \bbP(\tau_{Ct}\le t)\le \kappa(C)^t
\end{equation}
with $\kappa(C)<1$ if $C$ is large. On the other hand, symmetry of $u$
across each hyperplane $x=k\in\bbZ$ shows that $Y_j\equiv
(X_{\tau_j}^{A_n,x} - X_{\tau_{j-1}}^{A_n,x})_1$ are iids with $\bbP
\big( Y_j=\pm L\big) = \frac 12$. This gives
\[
\bbP(|X_{j(t)}^{A_n,x}-x|\ge L(Ct)^{9/17}\,\big|\, j(t)\le  Ct) \le
e^{-(Ct)^{\delta}}
\]
for some $\delta>0$ by
\begin{align*}
\sum_{k=0}^{\frac 12(j-j^{\delta+\frac 12})} \frac{{j\choose
k}}{2^j} & \approx (1+j^{\delta-\frac 12})^{-\frac
12(j+j^{\delta+\frac 12})}
(1-j^{\delta-\frac 12})^{-\frac 12(j-j^{\delta+\frac 12})} \\
& = \Big[ (1-j^{2\delta-1})^{-j^{1-2\delta}} (1+j^{\delta-\frac
12})^{-j^{\frac 12-\delta}} (1-j^{\delta-\frac 12})^{j^{\frac
12-\delta}} \Big]^{j^{2\delta}/2}  \approx e^{-j^{2\delta}/2},
\end{align*}
where we used Stirling's formula again. This, the fact that
$|X_{\tau_{j(t)}}^{A_n,x}-X_{t}^{A_n,x}|\le L$ (by the definition of
$\tau_j$ and $j(t)$), and \eqref{3.2} yield \eqref{3.1} for large
enough $t$ (with a different $\delta>0$).

We will also need the conclusion of Lemma 3.1 in \cite{FKR} which
says that there is $\til c>0$ such that for any $x\in D$, $m\in\bbZ$,
$A\in\bbR$, incompressible $u$, and $t\ge 1$ we have
\begin{equation} \lb{3.3}
\bbP \big( \big(X_t^{A,x}\big)_1\in [m,m+1] \big) \le \til
c t^{-1/2}.
\end{equation}
We note that \cite{FKR} only considers $d=2$, but the
general case is identical.

Let us now take non-negative $\psi_0\in C(\bbR)\cap C^3([-2,2])$
such that
\begin{gather}
{\rm supp}\, \psi_0=[-2,2], \notag\\
\psi_0(s)=\psi_0(-s) \hbox{$\qquad$ and $\qquad$}  \psi_0(0)=\tfrac {2+\tht_0}3, \notag\\
\psi_0(s) = \tfrac {1+\tht_0}6 \big[ (3-|s|)^2-1 \big] \text{ for $|s|\in[1,2]$}, \notag\\
\psi_0' \text{ is decreasing on $[-1,1]$}. \notag
\end{gather}
Note that this means that $\psi_0$ is non-negative, symmetric,
non-increasing on $\bbR^+$, and convex where $f(\psi_0(s))=0$. We then
let
\[
T_0(x)\equiv \psi_0 \left( \frac{x_1}M \right)\ge 0
\]
with a large $M\in\bbZ$ to be determined later. We will show using
the properties of $\psi_0$ that if $T$ solves \eqref{1.1} with
$A=A_n$, then for $\tau\equiv M^{3/2}$ we have
\begin{equation} \lb{3.4}
T(\tau,x)\ge T_0(x)
\end{equation}
(which gives the desired result by comparison theorems).

Let $\eps$ be such that $\psi_0(1+\eps)=\tfrac{1+2\tht_0}3$ and $M$ such
that $\eps M+ M^{4/5}\le M-2$. Let $\phi$ be the solution of
\eqref{2.2} with $\phi_0\equiv T_0$ and assume first that
$x_1\in[(1+\eps) M,2M-M^{4/5}]\cap\bbZ$. Let $x'\equiv
(x_2,\dots,x_d)$. Then by \eqref{2.3}, monotonicity of $\psi_0$ on
$\bbR^+$, and symmetry of $u$,
\begin{align}
\phi & (\tau,x)  \ge \sum_{m=-M^{4/5}-1}^{M^{4/5}}
\bbP((X_\tau^{A_n,x})_1\in[x_1+m,x_1+m+1])\phi_0(x_1+m+1,x') \notag\\
& = \sum_{m=0}^{M^{4/5}} \bbP((X_\tau^{A_n,x})_1\in [x_1+m,x_1+m+1])
(\phi_0(x_1+m+1,x')+\phi_0(x_1-m,x')). \lb{3.5}
\end{align}
We have
\[
\psi_0 \left( \frac{x_1+m+1}M \right) + \psi_0 \left( \frac{x_1-m}M
\right) = 2\psi_0 \left( \frac{x_1+\frac 12}M \right) + \psi_0''
\left( \frac{x_1+\frac 12}M \right) \left( \frac{m+\frac 12}M
\right)^2 + O\left( \left( \frac{m}M \right)^3\right),
\]
and $\tau=M^{3/2}$ together with \eqref{3.1} implies that the sum of
the $\bbP(\cdot)$ terms in \eqref{3.5} is larger than $\tfrac
12(1-e^{-\tau^\delta}) = \tfrac 12(1-e^{-M^{3\delta/2}})$. This and
$\psi_0''(s)=\tfrac{1+\tht_0}3$ for $s\in(1,2)$ yields
\[
\phi(\tau,x) \ge (1-e^{-M^{3\delta/2}}) \phi_0 \left( x_1+\tfrac
12,x'\right) + \frac{1+\tht_0}{12}(4\til cM^{1/4})^{-2} +
O(M^{-3/5}),
\]
where we also used that \eqref{3.3} gives
\[
\bbP \left(|X_\tau^{A_n,x}-x|\ge \frac{M^{3/4}}{4\til c} \right) \ge
\frac 12.
\]
Since $\phi_0(x)-\phi_0(x_1+\tfrac 12,x')=O(M^{-1})$, this means
\begin{equation} \lb{3.6}
\phi(\tau,x) \ge \phi_0(x) + c'M^{-1/2}
\end{equation}
for some $c'>0$ and any large enough $M$.

The same argument applies for any $\tau'\in[\tau/2,\tau]$ (with a
uniform $c'$) in place of $\tau$. This, Lemma \ref{L.2.3}, and the
fact that $\phi_0$ varies on a scale $O(M^{-1})$ on $[\lfloor x
\rfloor, \lfloor x \rfloor+1]\times\bbT^{d-1}$ yield
\eqref{3.6} for any $x_1\in[(1+\eps)M,2M-M^{4/5}]$, provided $M$ is
large enough. If $x_1\in[2M-M^{4/5},2M]$, then \eqref{3.6} follows
in the same way because $\psi_0(s)>\tfrac{1+\tht_0}6 [ (3-|s|)^2-1
]$ for $s\in(2,3)$. And if $x_1>2M$, then \eqref{3.6} is immediate
from $\phi(\tau,x)\ge 0$.

Symmetry and $T\ge\phi$ give \eqref{3.4} whenever $|x|\ge
(1+\eps)M$, so let us now consider $|x|\le (1+\eps)M$. As above we
obtain for large $M$,
\begin{equation} \lb{3.7}
\phi(\tau,x) \ge \phi_0(x) - c' M^{-1/2},
\end{equation}
where $c'$ only depends on $\|\psi_0''\|_\infty$. We now choose a
convex $g:\bbR^+\to\bbR^+$ with $g(s)\le f(s)$ for $s\le
\tfrac{3+\tht_0}4$ and $g(s)\ge \alpha$ for some $\alpha>0$ and all $s\ge
\tfrac{1+3\tht_0}4$. Define $\beta>0$ so that if $\gamma(0)=\tfrac{2+\tht_0}3$
and $\gamma'(s)=g(\gamma(s))$, then $\gamma(\beta)=\tfrac{3+\tht_0}4$.
Next let $\til f\equiv \tfrac\beta \tau g\le g$ when
$\tau=M^{3/2}\ge\beta$ and let $w:(\bbR^+)^2\to\bbR^+$ satisfy
$w(0,s)=s$ and
\[
w_t(t,s)=\til f(w(t,s)).
\]
Notice that
\begin{equation} \lb{3.8}
w(\tau,\tfrac{2+\tht_0}3)=\tfrac{3+\tht_0}4 \qquad \text{and} \qquad
w(\tau,s)\ge s+\alpha\beta \text{ for $s\ge \tfrac{1+3\tht_0}4$}.
\end{equation}
It is easy to show using $\til f',\til f''\ge 0$ that $w_s,w_{ss}\ge
0$. It then follows that $\til T(t,x)\equiv w(t,\phi(t,x))$ is a
sub-solution of \eqref{1.1} with $A=A_n$ and $\til T_0=T_0$
as long as $\|\til T(t,\cdot)\|_\infty \le \tfrac{3+\tht_0}4$ (so that
$\til f(\til T)\le f(\til T)$). Since $\|\phi\|_\infty\le
\psi_0(0)=\tfrac{2+\tht_0}3$, this is true for all $t\le \tau$ by
\eqref{3.8} and $w_t,w_s\ge 0$. But then $T(\tau,x)\ge \til
T(\tau,x)$, while large enough $M$ guarantees for $|x|\le
(1+\eps)M$,
\[
\phi(\tau,x) \ge \phi_0(x) - c'M^{-1/2} \ge \tfrac{1+2\tht_0}3 -
c'M^{-1/2} \ge \tfrac{1+3\tht_0}4.
\]
So for these $x$ by \eqref{3.8},
\[
T(\tau,x)\ge \til T(\tau,x) \ge \phi(\tau,x) +\alpha\beta \ge
\phi_0(x) - c'M^{-1/2} +\alpha\beta \ge \phi_0(x)=T_0(x)
\]
when $M$ is large. This is \eqref{3.4} and thus concludes the proof.
\end{proof}

\section{Proof of Theorem \ref{T.1.1}: Part II} \lb{S4}

We now assume that $u$ and $f$ are as in Theorem \ref{T.1.1} and
$A_n\to\infty$ is such that Lemma \ref{L.2.2}(i) holds. We will then
show that $\limsup_{n\to\infty}c_*(A_n)=\infty$, and that there is
$c>0$ such that if $f$ is of ignition type with
$\|f(s)/s\|_\infty\le c$, then any compactly supported initial datum
$T_0$ for \eqref{1.1} is quenched by some flow $A_nu$. The analysis
in this section applies in two dimensions only, so we will consider
$d=2$ and $D=\bbR\times\bbT$.

\begin{lemma} \lb{L.4.1}
Consider the setting of Theorem \ref{T.1.1} with $D=\bbR\times\bbT$
and let $A_n\to\infty$ be such that Lemma \ref{L.2.2}(i) holds. Then
$\limsup_{n\to\infty}c_*(A_n)=\infty$.
\end{lemma}

\begin{proof}
Assume that $c_*(A_n)\le c_0<\infty$ for all $n$ and let $T$ be a
pulsating front solution of \eqref{1.1} with $A=A_n$ and speed
$c_*(A_n)$, that is,
\begin{gather} \lb{4.1}
\begin{split}
T(t+c_*(A_n)^{-1},x_1+1,x_2)=T(t,x_1,x_2), \\
T(t,\pm\infty,x_2)=\tfrac 12 \mp \tfrac 12 \quad \text{ uniformly in
$x_2$}
\end{split}
\end{gather}
(recall that $u$ has period 1 in $x_1$). We note that \cite{BH}
shows
\begin{equation} \lb{4.2}
T_t(t,x)\ge 0.
\end{equation}
Integrating \eqref{1.1} over $[0,c_*(A_n)^{-1}]\times D$ and using
\eqref{4.1} and incompressibility of $u$, we obtain
\[
1=\int_0^{c_*(A_n)^{-1}} \int_D f(T(t,x))\,dxdt.
\]
Next we multiply \eqref{1.1} by $T$ and again integrate as  above to
get
\[
\frac 12 = \int_0^{c_*(A_n)^{-1}} \int_D T(t,x)f(T(t,x))-|\nabla
T(t,x)|^2\,dxdt \le 1 - \int_0^{c_*(A_n)^{-1}} \int_D |\nabla
T(t,x)|^2\,dxdt.
\]
This means that for some $t\in[0,c_*(A_n)^{-1}]$ (which we take to
be 0 by translating $T$ in time),
\begin{align}
\int_D f(T(0,x))\,dx \le 2c_0,  \lb{4.3}\\
\int_D |\nabla T(0,x)|^2\,dx \le c_0.  \lb{4.4}
\end{align}
We will now show that \eqref{4.1}--\eqref{4.4} force the reaction
zone (front width) to be bounded in the following sense. Let
$D^-_\eps$ be the rightmost cell $[m^-_\eps,m^-_\eps+1]\times\bbT$
such that $\inf_{x\in D^-_\eps} T(0,x)\ge 1-\eps$ (i.e., $m^-_\eps$
is the largest integer for which this condition holds). We also let
$D^+_\eps$ be the leftmost cell $[m^+_\eps,m^+_\eps+1]\times\bbT$
such that $\sup_{x\in D^+_\eps} T(0,x)\le 1-\eps$. Obviously
$m^-_\eps< m^+_\eps$. We will now show that for each small $\eps>0$
there is $L_\eps<\infty$ such that for each $n$ we have
\begin{equation} \lb{4.5}
m^+_{10\eps} - m^-_{\eps} \le L_\eps.
\end{equation}

Assume for a moment that \eqref{4.5} holds. Periodicity and
\eqref{2.7} tell us that there are $n$ and $x\in D^-_\eps$ such that
\[
\bbP \big(|X^{A_n,x}_\tau-x|\ge L_\eps \big) > \frac 12
\]
for $\tau\equiv \eps\|f'\|_\infty^{-1}>0$. Since $x_1\ge m^-_\eps\ge
m^+_{10\eps}-L_\eps$, symmetry of $u$ implies
\[
\bbP \big((X^{A_n,x}_\tau)_1\ge m^+_{10\eps} \big) > \frac 14.
\]
Using \eqref{2.4a} and \eqref{2.3} we have
\[
T(\tau,x)\le e^{\tau \|f'\|_\infty} \left( \frac 34 + \frac
{1-10\eps}4 \right) < 1-\eps \le T(0,x)
\]
if $\eps>0$ is small. This contradicts \eqref{4.2}, so our
assumption $c_*(A_n)\le c_0<\infty$ must be invalid. Thus the proof
will be finished if we establish \eqref{4.5} for all small $\eps>0$.

Let us consider an arbitrary small $\eps>0$ such that $f$ is bounded
away from zero on $[1-13\eps,1-\tfrac \eps 3]$ and assume, towards
contradiction, that for each $L\in\bbN$ there is $n$ such that
\begin{equation} \lb{4.6}
m^+_{10\eps} - m^-_{\eps} \ge 10L.
\end{equation}
Let $T_0(x)\equiv T(0,x)$,
\[
\bar T_0(x)\equiv \int_{[\lfloor x_1 \rfloor, \lfloor x_1 \rfloor
+1]\times\bbT} T_0(x)\,dx,
\]
and denote $D_j\equiv [m^-_\eps+j, m^-_\eps+j +1]\times\bbT$. Then
\eqref{4.4} and Poincar\'e inequality (with constant $C$) imply that
for each small $\delta>0$ and $L\equiv \lceil Cc_0/\delta\rceil$,
at least $7L$ of the cells $D_j$, $j=L,\dots,9L$, satisfy
\begin{equation} \lb{4.6a}
\|T_0-\bar T_0\|^2_{L^2(D_j)} \le C\|\nabla T_0\|^2_{L^2(D_j)} \le
\delta.
\end{equation}
Hence there are at least $\lfloor \tfrac{3L}5\rfloor$ disjoint
5-tuples of consecutive cells satisfying \eqref{4.6a}. Then
\eqref{4.3}, $f$ bounded away from zero on $[1-13\eps,1-\tfrac \eps
3]$, and $\bar T_0(D_j)$ decreasing in $j$ (by \eqref{4.2}) imply
that for some $j_0\in[L,9L]$ we must have either \eqref{4.6a} and
$\bar T_0(D_{j})\le 1-12\eps$ for $j=j_0-2,\dots,j_0+2$, or
\eqref{4.6a} and $\bar T_0(D_{j})\ge 1-\tfrac\eps 2$ for
$j=j_0-2,\dots,j_0+2$ (provided $\delta$ is small enough and $L$
large).

Let us assume the case $\bar T_0(D_{j})\le 1-12\eps$ for
$j=j_0-2,\dots,j_0+2$, $j_0\in[L,9L]$. Then \eqref{4.2} and
\eqref{4.6} say that there must be $y\in D_{j_0}$ such that for
$t\ge 0$,
\begin{equation} \lb{4.7}
T(t,y)\ge T_0(y)\ge 1-10\eps.
\end{equation}
Let $S^-_{2\gamma}\subset D_{j_0-2}\cup D_{j_0-1}\cup D_{j_0}$ be
the square of a small side $2\gamma>0$ (to be chosen later) centered
at $y^-\equiv y-(1,0)$ and denote by $\Gamma^-$ the intersection of
$S^-_{2\gamma}$ with the connected component $\Omega^-$ of the set
$\{x\,|\, T_0(x)\ge 1-11\eps\}$ containing $y^-$ (recall that that
$T_0(y^-)\ge T_0(y)\ge 1-10\eps$).

If $\Gamma^-$ has diameter less than $\gamma$ (in particular,
$\Gamma^-=\Omega^-\subseteq S^-_{2\gamma}$), then for
$\Gamma\equiv\Gamma^-+(1,0)$, all $x\in\partial \Gamma$, and all
$t\le c_*(A_n)^{-1}$,
\[
T(t,x)\le T(0,x-(1,0)) \le 1-11\eps
\]
by \eqref{4.1} and \eqref{4.2}. It follows by comparison that
$T(t,x)\le e^{t\|f'\|_\infty}(R(t,x) + 1-11\eps)$ where $R(t,x)$
solves \eqref{2.2} on $S_{2\gamma}\equiv S^-_{2\gamma}+(1,0)$ with
Dirichlet boundary conditions and $R(0,x)= 11\eps\chi_\Gamma(x)$.
But then the uniform bound in Lemma \ref{L.2.3} and parabolic
scaling in $(t,x)$ gives that for any $t>0$ there is small enough
$\gamma>0$ such that $\|R(t,x)\|_\infty\le \tfrac \eps 2$, and if
$t$ is chosen small enough (and $\gamma$ accordingly), then
$T(t,y)<1-10\eps$ follows. This clearly contradicts \eqref{4.7}.

If instead (for the chosen $\gamma$) the set $\Gamma^-\subset
D_{j_0-2}\cup D_{j_0-1}\cup D_{j_0}$ has diameter at least $\gamma$,
then $\bar T_0(D_{j})\le 1-12\eps$ and $\inf T_0(\Gamma^-)\ge
1-11\eps$ imply that the second inequality in \eqref{4.6a} must be
violated for at least one of $j=j_0-2,j_0-1,j_0$, provided
$\delta>0$ is chosen small enough (depending on $\gamma,\eps$).
Indeed --- if $\|\nabla T_0\|^2_{L^2(D_j)}$ is small enough, then
$T$ must be close to $1-11\eps$ on some vertical line passing
through $\Gamma^-$, and then $T$ must be close to $1-11\eps$ on most
horizontal lines inside $D_j$ by the same argument. This contradicts
$\bar T_0(D_{j})\le 1-12\eps$.

Finally, if we instead assume $\bar T_0(D_{j})\ge 1-\tfrac\eps 2$
for $j=j_0-2,\dots,j_0+2$ and $T(t,y)\le T_0(y-(1,0))\le 1-\eps$ for
small $t\ge 0$, a similar argument again leads to contradiction.
This means that \eqref{4.6} cannot hold for small $\eps>0$ and
\eqref{4.5} follows. The proof is finished.
\end{proof}

\begin{lemma} \lb{L.4.2}
Consider the setting of Theorem \ref{T.1.1} with $D=\bbR\times\bbT$.
There is $c>0$ such that if $f$ is of ignition type with
$\|f(s)/s\|_\infty\le c$ and $A_n\to\infty$ is such that Lemma
\ref{L.2.2}(i) holds, then for any compactly supported
$T_0(x)\in[0,1]$ there is $n$ such that the solution $T$ of
\eqref{1.1} with $A=A_n$ quenches.
\end{lemma}

\noindent
{\it Remark.} We note that $c$ is from Lemma \ref{L.2.3} and can be
easily evaluated from its proof.

\begin{proof}
By comparison theorems, it is sufficient to consider initial data
$T_0(x)\equiv\chi_{[-L,L]}(x_1)$ for all $L\in\bbN$. Let $\phi$ be
the solution of \eqref{2.2} with $A=A_n$ and initial datum
$\phi_0\equiv T_0$. We first claim that for each $\tau,\delta>0$
there is $n$ and a continuous curve $h:[0,1]\to [0,1]\times\bbT$
such that $(h(0))_1=0$ and $(h(1))_1=1$ , and for all $s\in[0,1]$
and $t\ge\tau$,
\begin{equation} \lb{4.8}
\phi(t,h(s)) \le \delta.
\end{equation}

To this end we let $\psi$ be the solution of \eqref{2.2} with
initial condition $\psi_0\equiv\chi_{[-K-2,K]}(x_1)$ where $K\ge
3L\delta^{-1}$. By periodicity of $u$ and \eqref{2.7}, there must be
$n$ (which will be kept constant from now on) and
$y\in[-1,0]\times\bbT$ such that
\[
\psi(\tau,y)=\bbP \big((X^{A_n,y}_\tau)_1\in [-K-2,K] \big) \le
\frac\delta 2.
\]
The maximum principle for \eqref{2.2} implies that the connected
component of the set
\[
\{(t,x)\in[0,\tau]\times D \,|\,\psi(t,x)\le\tfrac\delta 2\}
\]
containing $(\tau,y)$ must intersect
\[
\{x\in D \,|\,\psi(0,x)\le\tfrac\delta 2\} =
(\bbR\setminus[-K-2,K])\times\bbT.
\]
Since by symmetry $\psi(t,x_1,x_2)=\psi(t,-2-x_1,x_2)$ for $x_1\ge
0$, this means that there is a curve $h(s)$ joining
$\{0\}\times\bbT$ and $\{K\}\times\bbT$ such that for each $s$ there
is $\tau_s\le\tau$ with 
\[
\psi(\tau_s,h(s))=\bbP \big((X^{A_n,h(s)}_{\tau_s})_1\in [-K-2,K]
\big) \le \frac\delta 2.
\]
Lemma \ref{L.2.1}(iii) and the definition of $K$ then mean that for
all $t\ge\tau$,
\[
\phi(t,h(s))=\bbP \big(|X^{A_n,h(s)}_t|\le L] \big) \le \frac\delta
2 + \Big(1-\frac\delta 2\Big) \frac\delta 3 \le \delta
\]
which is \eqref{4.8} (after reparametrization of $h$ and restriction
to $s\in[0,1]$).

Symmetry of $u$ and $\phi_0$ implies that \eqref{4.8} holds for
$h(s)$ extended to $s\in[-1,1]$ by $h(-s)=(-(h(s))_1,(h(s))_2)$.
Finally, \eqref{4.8} applies to $h(s)$ extended periodically (with
period 2) onto $\bbR$. This last claim holds because
$\phi(t,x)\ge\phi(t,x+(2,0))$ when $x_1\ge -1$ (and
$\phi(t,x)\ge\phi(t,x-(2,0))$ when $x_1\le 1$), which in turn
follows because $\phi(t,x)-\phi(t,x+(2,0))$ solves \eqref{2.2} with
initial datum that is symmetric across $x_1=-1$ and non-negative on
$[-1,\infty)\times\bbT$ (and hence stays such by the symmetry of $u$).

This means that $\|\phi(t+\tau,\cdot)\|_\infty\le
\|\psi(t,\cdot)\|_\infty+\delta$ where $\psi$ is the solution of
\eqref{2.2} on $2\bbT\times\bbT$ with $\psi_0\equiv 1$ and
$\psi(t,h(s))=0$ for all $t>0$ and $s\in[0,2]$.
Since the Poincar\' e inequality and the proof of Lemma \ref{L.2.3}
extend to this setting with the same universal constant $c>0$, we
obtain that
$\|\phi(t,\cdot)\|_\infty\le \delta + 2e^{-c(t-\tau)}$. If now
$\|f(s)/s\|_\infty = c' < c$ and $\tau,\delta>0$ are chosen small
enough depending on $c-c'$ (and $n$ accordingly), we obtain
$\|T(t_0,\cdot)\|_\infty\le e^{c't_0}(\delta+2e^{c\tau}e^{-ct_0})\le
\tht_0$ for some $t_0$. The maximum principle then implies
$\|T(t,\cdot)\|_\infty\le\tht_0$ for any $t\ge t_0$ and quenching
follows.
\end{proof}

The proof of Theorem \ref{T.1.1} is now based on the last four
lemmas and this result from \cite{RZ}:

\begin{lemma} \label{L.5.1}
Assume the setting of Theorem \ref{T.1.1} with $f$ a KPP
nonlinearity and $D=\bbR\times\bbT$.
\begin{SL}
\item[{\rm{(i)}}] If \eqref{1.2} on $2\bbT\times\bbT$
has a solution $\psi\in H^1(2\bbT\times\bbT)$, then \eqref{1.3}
holds.
\item[{\rm{(ii)}}] If \eqref{1.2} has no $H^1(2\bbT\times\bbT)$-solutions,
then \eqref{1.4} holds.
\end{SL}
\end{lemma}

\begin{proof}[Proof of Theorem \ref{T.1.1}]
If \eqref{1.2} has a solution $\psi\in H^1(2\bbT\times\bbT)$, then
$c_*(A_n)$ is bounded for any KPP $f$ and any $A_n\to\infty$, and so
Lemma \ref{L.4.1} gives Lemma \ref{L.2.2}(ii). Lemmas \ref{L.3.1}
and \ref{L.3.2} now give (i) for any $f$. Note that if each sequence
$A_n$ does not quench some compactly supported initial datum $T_0$
for \eqref{1.1} with $A=A_n$, then there is $T_0$ that is not
quenched by any $A$. This holds because if each
$T_0(x)\equiv\chi_{[-n,n]}(x_1)$ is quenched by some $A_n$, then
this sequence would yield a contradiction.

If, on the other hand, \eqref{1.2} has no
$H^1(2\bbT\times\bbT)$-solutions, then $c_*(A_n)\to\infty$ for any
KPP $f$ and any $A_n\to\infty$, and so Lemma \ref{L.3.1} gives Lemma
\ref{L.2.2}(i). Lemma \ref{L.4.1} now gives \eqref{1.4} for any $f$.
The claim about the existence of $l_0$ follows from the fact that
$T$ solves $T_t- Au^{(l)}\cdot\nabla T = \Delta T + f(T)$ on
$\bbR\times l\bbT$ if and only if $S(t,x)\equiv T(l^2t,lx)$ solves
$S_t- Alu\cdot\nabla S =\Delta S + l^2f(S)$ on $\bbR\times \bbT$.
Comparison theorems and $f\ge 0$ then show that if $u^{(l)}$ is
quenching for $f$, then so is $u^{(\til l)}$ for any $\til l<l$.
This only guarantees $l_0\in[0,\infty]$, but $l_0<\infty$ follows
from Theorem 8.2 in \cite{ZlaMix} and the fact that the flow $u$
leaves the bounded domain $[0,p]\times\bbT$ invariant. For ignition
reactions Lemma \ref{L.4.2} shows $l_0>0$ --- if each $T_0$ is
quenched by at least one $A_nu$ for any sequence $A_n\to\infty$,
then each $T_0$ is quenched by $Au$ for all large $A$.
\end{proof}

Finally, we provide the following extension of Theorem \ref{T.1.1}(ii) to
some positive reactions.

\begin{corollary} \label{C.5.2}
The claim $l_0>0$ in Theorem \ref{T.1.1}(ii) holds for any
combustion-type reaction satisfying $f(s)\le \alpha s^{\beta}$ for
some $\alpha>0$, $\beta>3$, and all $s\in [0,1]$.
\end{corollary}

\begin{proof}
%
By the proof of Theorem \ref{T.1.1}, it is sufficient to show that
there is $l>0$ such that $u$ is quenching for $l^2f(s)$. The proof
is essentially identical to that of Theorem 8.3 in \cite{ZlaMix}. We
let $I_A\equiv \int_0^\infty \|\phi(t,\cdot)\|_\infty^{\beta-1}\,dt$
where $\phi$ is the solution of \eqref{2.2} and $\phi_0(x)\equiv
T_0(x)$. It follows from \cite{Me} (see also \cite[Lemma
2.1]{ZlaArrh}) that $u$ is quenching for $l^2f(s)$ when for each
compactly supported $T_0$ there is $A_0$ such that
$l^2\alpha(\beta-1)I_{A}<1$ whenever $A\ge A_0$. So fix $T_0$ and
notice that the bound $\|\phi(t,\cdot)\|_\infty\le \til c |\supp\,
T_0| t^{-1/2}$ for $t\ge 1$, which follows from \eqref{3.3}, gives
$\int_{t_0}^\infty \|\phi(t,\cdot)\|_\infty^{\beta-1}\,dt\le 1$ if
$t_0$ is chosen appropriately (depending on $\til c|\supp\, T_0|$).
For $t\le t_0$ we use the bound $\|\phi(t,\cdot)\|_\infty\le
5e^{-ct}$, which follows from the proof of Lemma \ref{4.2} (with the
same $c$) provided $A_0$ is chosen large enough so that $\delta$ in
that proof is smaller than $e^{-ct_0}$ for each $A\ge A_0$ (and
$\tau$ is such that $e^{c\tau}\le 2$). This choice is possible
because each sequence $A_n\to\infty$ has a term $A_n$ guaranteeing
$\delta<e^{-ct_0}$. Hence for $A\ge A_0$ we have
\[
\int_0^{t_0} \|\phi(t,\cdot)\|_\infty^{\beta-1}\,dt \le
\int_0^{\infty} (5e^{-ct})^{\beta-1}\,dt \equiv C <\infty.
\]
Now let $l>0$ be such that $l^2\alpha(\beta-1)(1+C)<1$, and we are
done.
\end{proof}


\end{document}